\documentclass[11pt]{amsart}
\usepackage{geometry}                
\usepackage{amssymb}
\usepackage{amsthm}
\usepackage{epsfig}
\usepackage{graphics,color}

\numberwithin{equation}{section}

\newtheorem{theorem}{Theorem}[section]
\newtheorem{lemma}[theorem]{Lemma}

\theoremstyle{definition}

\newtheorem{remark}[theorem]{Remark}

\newtheorem{open}[theorem]{Open problem}

\newcommand\C{{\operatorname{Cov}}}
\newcommand\Cov{\operatorname{Cov}}
\newcommand\gnp{\ensuremath{G(n,p)}}
\newcommand\ggnp{{\ensuremath{\vec G(n,p)}}}

\usepackage{epic}
\usepackage{eepic}
\usepackage{tikz}
\usetikzlibrary{decorations.markings,arrows}

\tikzstyle{block}=[draw opacity=0.7,line width=1.4cm]

\newcommand\XA{X_A}
\newcommand\XB{X_B}
\newcommand\YA{Y_A}
\newcommand\YB{Y_B}
\newcommand\ZA{Z_A}
\newcommand\xa{\XA}
\newcommand\xb{\XB}
\newcommand\xaa{X_A'}
\newcommand\xbb{X_B'}
\newcommand\yaa{Y_A'}
\newcommand\ybb{Y_B'}
\newcommand\IA{I_A}
\newcommand\IB{I_B}
\newcommand\iaa{I_A'}
\newcommand\ibb{I_B'}
\newcommand\ja{J_A}
\newcommand\jb{J_B}
\newcommand\la{{\ell_\ga}}
\newcommand\lax{\ell_\ga}
\newcommand\lb{{\ell_\gb}}
\newcommand\lba{\gd}
\newcommand\uv{\mu}
\newcommand\lab{{\ell_{\ga\gb}}}
\newcommand\gab{\ensuremath{(\ga,\gb)}}
\newcommand\ia{I_\ga}
\newcommand\ib{I_\gb}
\newcommand\ig{I_\gam}
\newcommand\ggl[1]{\gG^L_{#1}}
\newcommand\nn[1]{n^{-#1}}
\newcommand\onn[1]{O\bigpar{n^{-#1}}}

\newcommand{\refT}[1]{Theorem~\ref{#1}}

\newcommand{\refL}[1]{Lemma~\ref{#1}}

\newcommand{\refS}[1]{Section~\ref{#1}}
\newcommand{\refF}[1]{Figure~\ref{#1}}
\newcommand{\refand}[2]{\ref{#1} and~\ref{#2}}

\newcommand\ga{\alpha}
\newcommand\gb{\beta}
\newcommand\gd{\delta}

\newcommand\gam{\gamma}
\newcommand\gG{\Gamma}

\newcommand\E{\operatorname{\mathbb E{}}}
\renewcommand\P{\operatorname{\mathbb P{}}}
\renewcommand{\=}{:=}

\newcommand\bigpar[1]{\bigl(#1\bigr)}
\newcommand\Bigpar[1]{\Bigl(#1\Bigr)}

\newcommand\lrpar[1]{\left(#1\right)}
\newcommand\parfrac[2]{\lrpar{\frac{#1}{#2}}}

\newcommand\Bigparfrac[2]{\Bigpar{\frac{#1}{#2}}}

\newcommand\punkt{\spacefactor=1000}  
    
\newcommand\ie{i.e\punkt}

\newcommand\set[1]{\ensuremath{\{#1\}}}

\newcommand\qw{^{-1}}
\newcommand\qww{^{-2}}

\newcommand\qwq{^{-99}}

\newenvironment{romenumerate}[1][0pt]{
\addtolength{\leftmargini}{#1}\begin{enumerate}
 \renewcommand{\theenumi}{\textup{(\roman{enumi})}}%
 }{\end{enumerate}}

\newenvironment{romx}{%
\setcounter{enumi}{0}%
\renewcommand\item[1]{
\smallskip\noindent\refstepcounter{enumi}\emph{\theenumi{} ##1}:}%
 \renewcommand{\theenumi}{\textup{(\roman{enumi})}}%
 }{}

\newcommand\REM[1]{{\raggedright\texttt{[#1]}\par\marginal{XXX}}}

\begin{document}

\title{First critical probability for a problem on random orientations in
  $G(n,p)$.}

\author{Sven Erick Alm} 
\address{Department of Mathematics, Uppsala University, 
P.O.\ Box 480,  SE-751 06, Uppsala, Sweden.}
\email{sea@math.uu.se}

\author{Svante Janson} 
\thanks{Svante Janson is supported by the Knut and Alice Wallenberg Foundation}
\address{Department of Mathematics, Uppsala University, 
P.O.\ Box 480,  SE-751 06, Uppsala, Sweden.}
\email{svante.janson@math.uu.se}
\author{Svante Linusson} \thanks{Svante Linusson is a Royal Swedish Academy of Sciences Research Fellow supported by a grant from 
the Knut and Alice Wallenberg Foundation.}
\address{Department of Mathematics, KTH-Royal Institute of Technology, 
  SE-100 44, Stockholm, Sweden.}
\email{linusson@math.kth.se}
\thanks{This research was initiated when all three authors visited the
  Institut Mittag-Leffler (Djursholm, Sweden).}

\date{25 March, 2013} 

\begin{abstract}
We study the random graph $\gnp$ with a random orientation. For three fixed
vertices $s,a,b$ in 
$G(n,p)$ we study the correlation of the events $\{a\to s\}$ and $\{s\to b\}$. 
We prove that asymptotically the correlation is negative for small $p$,
$p<\frac{C_1}n$, where $C_1\approx0.3617$, positive for 
$\frac{C_1}n<p<\frac2n$ and up to $p=p_2(n)$. Computer aided computations
suggest that $p_2(n)=\frac{C_2}n$, with $C_2\approx7.5$.  
We conjecture that the correlation  then stays negative for $p$ up to the
previously known zero at $\frac12$; for larger $p$ it is positive.
\end{abstract}

\maketitle

\section{Introduction} \label{S:Intro}
Let $G(n,p)$ be the random graph with $n$ vertices where each edge has
probability $p$ of being present independent of the other edges.  
We further orient each present edge either way independently with
probability $\frac{1}{2}$, and denote the resulting random directed graph by
\ggnp.  
This version of orienting edges in a graph, random or not, is natural and
has been considered previously in e.g. \cite{AL,AJL,G00,M81}. 

Let $a,b,s$ be three distinct vertices and define the events $A:=\{a\to
s\}$, that there exists a directed path in $\ggnp$ from $a$ to $s$, and 
$B:=\{s\to b\}$.
In a previous paper, \cite{AJL}, we showed that, for fixed $p$, the
correlation between $A$ and $B$ asymptotically is negative for  
$p<\frac12$ and positive for $p>\frac12$. 
Note that we take the covariance
in the combined  probability space of 
$\gnp$ and the orientation of edges, which is often referred to as the
annealed case, see \cite{AJL} for details.  
We say that a probability $p\in(0,1)$ is \emph{critical} (for a given $n$)
if the covariance $\C(A,B)=0$. 
We have thus shown in \cite{AJL} that there is a critical
probability $\frac12+o(1)$ for large $n$.
(Moreover, this is the largest critical probability, since the covariance
stays positive for all larger $p<1$.)
We also conjectured 
that for large $n$,
there are in fact (at least) three critical probabilities when the
covariance changed sign. Based on computer aided computations we guessed
that the first two critical probabilities would be approximately
$\frac{0.36}{n}$ 
and $\frac{7.5}{n}$. In this note we prove that there is a first critical
probability of the conjectured order, where the covariance changes from negative to positive, and thus there must be at least three
critical probabilities.
Our theorem is as follows.

\begin{theorem}\label{T:firstzero}
With $p=\frac{2c}n$ and sufficiently large $n$, 
the covariance $\C(A,B)$ is negative for $0<c<c_1$ and positive for $c_1<c<1$, 
where $c_1\approx0.180827$ is a solution to
$(2-c)(1-c)^3=1$.
Furthermore, for fixed $c$ with $0\le c<1$,
\begin{equation}
  \label{main}
\C(A,B)
=\left(1-(2-c)(1-c)^3\right)\cdot \frac{c^3}{(1-c)^5}\cdot \frac1{n^3}
+O\Big(\frac1{n^4}\Big).
\end{equation}
\end{theorem}
In fact, the proof shows that \eqref{main} holds uniformly in $0\le c\le c'$
for any $c'<1$; moreover, 
we may (with just a little more care) for such $c$ 
write the error term as $O(c^4n^{-4})$.
This implies that for large $n$, the critical $p\approx2c_1/n$ is indeed the
first critical probability, and that the covariance is negative for all smaller
$p>0$. 

\begin{remark}
In a random orientation of any given graph $G$, it is a fact first observed by McDiarmid
 that $\P(a\to s)$ is equal to $\P(a\leftrightarrow s)$ in an edge percolation on the same graph with probability $1/2$ for each edge independently, see \cite{M81}. Hence the events $A$ (and thus $B$) have the same probability as
$\P(a\leftrightarrow s)$ in $G(n,p/2)$. With $p=2c/n$ it is well known that for $c<1$  
this probability is $\frac{c}{(1-c)}n\qw+O(n\qww)$, see e.g.\ \cite{JL}. 
Hence the covariance in \eqref{main} is of the order $O(\P(A)\P(B)/n)$.
\end{remark}

The outline of the proof is as follows, see Sections \refand{S:pf}{S:Lemmas}
for details. 

Let $p:=2c/n$, where $c<1$.
Let $\XA:=\#\{a\to s\}$ be the number of paths from $a$ to $s$ in \ggnp{}
and $\XB:=\#\{s\to b\}$. 
(In the proof below, for technical reasons, we actually
only count paths that are not too long.)
We first show that, in our range of $p$, the probability that $\XA\ge2$ or
$\XB\ge2$ is small, and that we can ignore these events and approximate
$\C(A,B)$ by $\C(\xa,\xb)$. The latter covariance is a double sum over pairs
of possible paths \gab, where $\ga$ goes from $a$ to $s$ and $\gb$ goes from
$s$ to $b$, and we show that the largest contribution comes from
configurations  of the following two types:

\begin{description}
\item [Type 1]
The two edges incident to $s$, \ie{} the last edge in $\ga$ and the first
edge in $\gb$, are the same but with opposite orientations; all other edges
are distinct.
See Figure \ref{F:Type1}.
\begin{figure}[htbp]
\includegraphics[width=5cm]{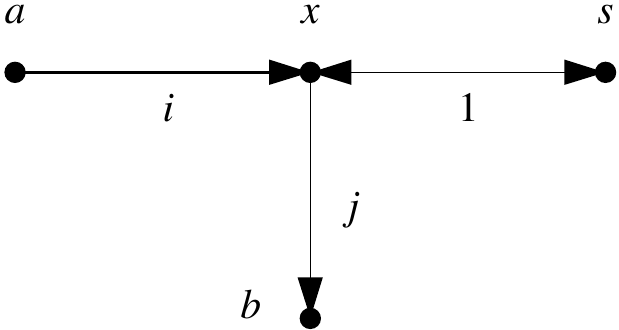}
\caption{Configurations of Type 1 ($i,j\ge0$, $i+j\ge1$).} 
\label{F:Type1}
\end{figure}

\item[Type 2]
$\ga$ and $\gb$ contain a common subpath with the same orientation, but
all other edges are distinct.
See Figure \ref{F:Type2}.
\begin{figure}[htbp]
\includegraphics[width=7cm]{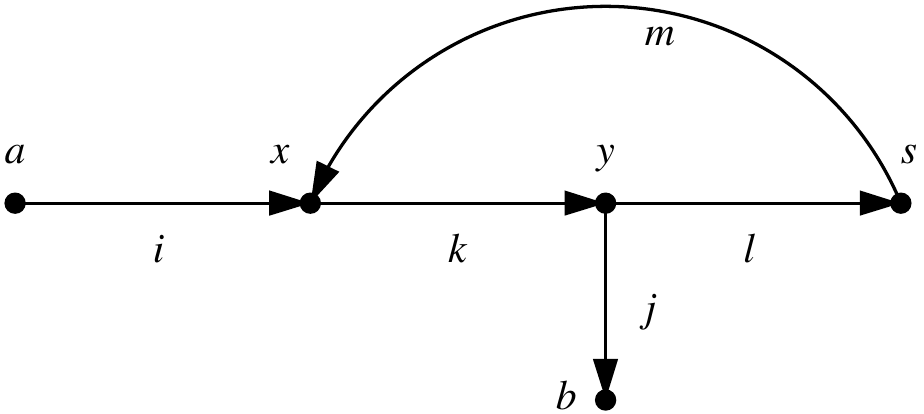}
\caption{Configurations of Type 2 ($i,j\ge0$, $k,l,m\ge1$).} 
\label{F:Type2}
\end{figure}

\end{description}

If \gab{} is of Type 1, then $\ga$ and $\gb$ cannot both be paths in \ggnp,
since they  contain an edge with opposite orientations. Thus each such pair
\gab{} gives a negative contribution to $\C(\XA,\XB)$.
Pairs of Type 2, on the other hand, give a positive contribution. It turns
out that both contributions are of the same order $\nn3$, see Lemmas
\refand{L:Type1}{L:Type2},
with constant factors depending on $c$ such that the negative contribution
from Type 1 dominates for small $c$, and the positive contribution from Type
2 dominates for larger $c$.

\begin{open}
It would be interesting to find a method to compute also the second critical
probability, which we in \cite{AJL} conjectured to be approximately
$\frac{7.5}{n}$. 
(The methods in the present paper apply only for $c<1$.)
Even showing that the covariance is negative when 
$p$ is of the order $\frac{\log n}{n}$ is open.
Moreover we conjecture that (for large $n$ at least) there are only three
critical probabilities, but that too is open.
\end{open}

\section{Proof of Theorem \ref{T:firstzero}}\label{S:pf}
We give here the main steps in the proof of \refT{T:firstzero},
leaving details to a sequence of lemmas in Section \ref{S:Lemmas}.

By a \emph{path} we mean a directed path $\gam=v_0e_1\cdots e_\ell v_\ell$ in
the complete graph $K_n$. 
We use the conventions that a path is self-avoiding,
i.e.\ has no repeated vertex, and that the length $|\gam|$
of a path is the number of edges in the path. 

We let $\gG$ be the set of all such paths and let, for two distinct vertices
$v$ and $w$, $\gG_{vw}$ be the subset of all paths from $v$ to $w$.

If $\gam\in\gG$, let $\ig$ be the indicator that $\gam$ is a path in \ggnp,
\ie{}, that all edges in $\gam$ are present in $\ggnp$ and have the correct
orientation there. Thus
\begin{equation}
  \label{a1}
\E\ig=\P(\ig=1)=\parfrac{p}{2}^{|\gam|}
= \parfrac{c}{n}^{|\gam|}.
\end{equation}

Let $\IA$ and $\IB$ be the indicators of $A$ and $B$. Note that the event
$A$ occurs if and only if
$\sum_{\ga\in\gG_{as}}\ia\ge1$, and similarly for $B$.

It will be convenient to restrict attention to paths that are not too long,
so we introduce a cut-off $L\=\log^2n$ and let $\ggl{vw}$ be the set of
paths in $\gG_{vw}$ of length at most $L$.
Let
\begin{align*}
  \xa\=\sum_{\ga\in\ggl{as}} \ia
&&
\text{and}
&&&
  \xb\=\sum_{\gb\in\ggl{sb}} \ib,
\end{align*}
\ie, the numbers of paths in \ggnp{} from $a$ to $s$ and from $s$ to $b$,
ignoring paths of length more than $L$.

Write $\XA=\iaa+\xaa$ and $\XB=\ibb+\xbb$, where
$\iaa$ and $\ibb$ are the indicators for the events $\xa\ge1$ and
$\xb\ge1$ respectively, so that 
\begin{align*}
  \iaa&=\min(\xa,1),
\\
\xaa&=(\xa-1)_+
=\begin{cases}0&\text{if }\XA\le1,\\
\XA-1&\text{if }\XA>1.\end{cases}
\end{align*}

We have $\IA\ge\iaa$. Let $\ja\=\IA-\iaa$
and $\jb\=\IB-\ibb$.
Thus
\begin{equation}\label{ijsplit}
\C(A,B)=\C(\IA,\IB)=\C(\iaa,\ibb)+\C(\iaa,\jb)+\C(\ja,\IB).
\end{equation}
We will show in \refL{L:L} below that the last terms are small:
$O(n\qwq)$. (The exponent 99 here and below can be replaced by any fixed
number.) 

Similarly, since $\iaa=\xa-\xaa$,
\begin{equation}\label{E:split}
\C(\iaa,\ibb)=\C(\xa,\xb)-\C(\xa,\xbb)-\C(\xaa,\xb)+\C(\xaa,\xbb),
\end{equation}
where \refL{L:rest} shows that the last three terms are $O(\nn4)$.
Hence, it suffices to compute
\begin{equation}
  \label{a3c}
\C(\XA,\XB)=\C\Bigpar{\sum_{\ga\in\ggl{as}}\ia,\sum_{\gb\in\ggl{sb}}\ib}
=\sum_{\ga\in\ggl{as}}\sum_{\gb\in\ggl{sb}}\Cov(\ia,\ib).
\end{equation}
Lemmas \ref{L:Type1} and \ref{L:Type2} yield the contribution to this sum
from pairs \gab{} of Types 1 and 2, and \refL{L:3} shows that the remaining
terms contribute only $O(\nn4)$. Using \eqref{ijsplit}--\eqref{a3c} and
the lemmas in \refS{S:Lemmas} we thus obtain
\begin{equation*}
\begin{split}
\C(A,B)&=\C(\iaa,\ibb)+O(n\qwq)
=\C(\xa,\xb)+O(\nn4)
\\&
=\left(-\frac{2c^3-c^4}{(1-c)^2}+\frac{c^3}{(1-c)^5}\right)\frac1{n^3}
+O\Big(\frac1{n^4}\Big),
\\&=\frac{c^3}{(1-c)^5}\cdot\Bigpar{1-(2-c)(1-c)^3}\frac1{n^3}
+O\Big(\frac1{n^4}\Big),
\end{split}  
\end{equation*}
which is \eqref{main}.

The polynomial $1-(2-c)(1-c)^3=-c^4+5c^3-9c^2+7c-1$ is negative for $c=0$
and has two real zeros, for example because its discriminant is $-283 <0$, see e.g.\ \cite{W1898}; a numerical calculation
yields the roots $c_1\approx0.180827$ and $c_2\approx2.380278$, 
which completes the proof.
\qed

\section{Lemmas}\label{S:Lemmas}

We begin with some general considerations.
We assume, as in \refT{T:firstzero}, that $p=2c/n$ and  $0\le c<1$.

Consider a term $\C(\ia,\ib)$ in \eqref{a3c}.
Suppose that $\ga$ and $\gb$ have lengths $\la$ and $\lb$.
Furthermore, suppose that $\gb$ contains $\lba\ge0$ edges not in $\ga$
(ignoring the orientations) and that these form $\uv\ge0$ subpaths of
$\gb$ that intersect $\ga$ only at the endvertices. 
(We will use the notation $\gb\setminus\ga$ for the set of (undirected)
edges in $\gb$ but not in $\ga$.) 
The number $\lab$ of edges common to $\ga$ and $\gb$ 
(again ignoring orientations)
is thus $\lb-\lba$.
By \eqref{a1},
$\E\ia=(c/n)^{\la}$
and
$\E\ib=(c/n)^{\lb}$.

\begin{romenumerate}  
\item 
If $\ga$ and $\gb$ have no common edge, then $\ia$ and $\ib$ are independent
and 
\begin{equation}\label{cov0}
\C(\ia,\ib)=0.  
\end{equation}

\item 
If
all common edges have the same
orientation in $\ga$ and $\gb$, then
\begin{equation}\label{cov+}
\C(\ia,\ib)=  \E(\ia\ib)-\E\ia\E\ib=\parfrac{c}{n}^{\la+\lba}
-\parfrac{c}{n}^{\la+\lb}.
\end{equation}

\item 
If some common edge has different orientations in $\ga$ and
$\gb$, then
$\E(\ia\ib)=0$
and
\begin{equation}\label{cov-}
  \C(\ia,\ib)=-\E\ia\E\ib
=-\parfrac{c}{n}^{\la+\lb}.
\end{equation}
\end{romenumerate}

We denote the falling factorials by $(n)_\ell\=n(n-1)\dotsm(n-\ell+1)$.
Note that the total number of paths of length $\ell$ in $\gG_{vw}$ is
$(n-2)_{\ell-1}\=(n-2)\dotsm(n-\ell)$, since the path is determined by
choosing $\ell-1$ internal vertices in order, and all vertices are distinct.

\begin{lemma}
  \label{L:L}
$\Cov(\iaa,\jb)=O(n\qwq)$ and $\Cov(\ja,\IB)=O(n\qwq)$.
\end{lemma}

\begin{proof}
$\ja$ is
  the indicator of the event that there is a path in \ggnp{} from $a$ to
  $s$, and that every such path has length $>L=\log^2 n$.
Thus, 
\begin{equation*}
  0\le\ja\le\sum_{\ga\in\gG_{as},\,|\ga|>L}\ia
\end{equation*}
and thus, using \eqref{a1} and the fact that there are $(n-2)_{\ell-1}\le
n^{\ell-1}$ paths of length $\ell$ in $\gG_{as}$,
\begin{equation*}
  0\le\E\ja\le
\sum_{\ga\in\gG_{as},\,|\ga|>L}\parfrac{c}{n}^{|\ga|}
\le\sum_{\ell=L}^\infty n^{\ell-1}\parfrac{c}{n}^{\ell}
\le\sum_{\ell=L}^\infty c^{\ell}
=O(c^L)=O(n\qwq).
\end{equation*}
Since $\ja,\ib\in[0,1]$,
\begin{equation*}
  |\C(\ja,\IB)|\le\E(\ja\IB)+\E\ja\E\IB\le2\E\ja=O(n\qwq).
\end{equation*}
Similarly,
$|\C(\iaa,\jb)|=O(n\qwq)$.
\end{proof}

\begin{lemma}\label{L:Type1}
Pairs of Type 1 contribute
$-\frac{1}{n^3}\frac{2c^3-c^4}{(1-c)^2}+O(\frac1{n^4})$ to the covariance
$\C(\XA,\XB)$.
\end{lemma} 
\begin{proof}
Let the path $\ga$ from $a$ to $s$ consist of $i+1$ edges, where the last
edge is the first in the path $\gb$ of length $j+1$ from $s$ to $b$, see
Figure \ref{F:Type1}.
The paths must not share any more edges, but could have more common
vertices. 
Here $i,j\ge0$ and $i+j\ge1$ since $a\neq b$.
Let $R_{i,j}$ be the number of such pairs of paths, for 
given $i$ and $j$. 
If $j\ge1$, 
the paths are determined by the choice of $i$ distinct
vertices for $\ga$ 
and then $j-1$ distinct vertices for $\gb$;
if $j=0$, then $i\ge1$ and the paths are determined by the choice of $i-1$
distinct vertices for $\ga$.
Order is important so, for $i,j\le L$, 
with a minor modification if $j=0$, 
\begin{equation*}
(n-2)_{i}\cdot (n-3)_{j-1}\ge R_{i,j}\ge
(n-2)_{i+j-1}, 
\end{equation*}
Thus $R_{i,j}= n^{i+j-1}\Bigpar{1+O\Bigparfrac{(i+j)^2}n}$ and
summing over all such pairs \gab{} gives by \eqref{cov-}
a contribution to $\C(\XA,\XB)$ of
\begin{align*}
-\sum_{i+j\ge 1}&R_{i,j}\Big(\frac cn\Big)^{i+j+2}
=
-\sum_{\substack{i+j\ge 1\\i,j\le L}}
n^{i+j-1}\Bigpar{1+O\Bigparfrac{(i+j)^2}n}
\Big(\frac cn\Big)^{i+j+2}
\\&
=
-\sum_{\substack{i+j\ge 1\\i,j\le L}} c^{i+j+2}\nn3
+\sum_{i+j\ge 1}O\bigpar{(i+j)^2} c^{i+j+2}\nn4
\\&
=-\nn3\Bigpar{
2\sum_{j\ge 1}c^{j+2}+\sum_{i,j\ge 1}c^{i+j+2}+O(c^L)}+\onn4
\\
&=-\nn3\left(\frac{2c^3}{1-c}+\frac{c^4}{(1-c)^2}\right)+\onn4
=-\nn3\cdot\frac{2c^3-c^4}{(1-c)^2}+\onn4.
\qedhere
\end{align*}
\end{proof}

\begin{lemma}\label{L:Type2}Type 2 pairs contribute
  $\frac1{n^3}\cdot\frac{c^3}{(1-c)^5}+O\big(\frac1{n^4}\big)$ to the
  covariance $\C(\XA,\XB)$.
\end{lemma}

\begin{proof}
A pair \gab{} of paths of Type 2 must contain a directed cycle containing
$s$, from which there are $m\ge1$ edges to a vertex $x$ to which there is a  
directed path of length $i\ge 0$ from $a$. 
The cycle continues from $x$ with $k\ge1$ edges to a vertex $y$, which
connects to $b$ via $j\ge0$ edges. The cycle is completed by $l\ge1$ edges
from $y$ to $s$, see Figure \ref{F:Type2}. 
By \eqref{cov+}, then
\begin{equation}
  \C(\ia,\ib)=\Bigparfrac cn^{i+j+k+l+m}\lrpar{1-\Bigparfrac cn^k}.
\end{equation}

Let $R_{i,j,k,l,m}$ be the number of such pairs \gab{} with given
$i,j,k,l,m$.
The path $\ga$ is determined by $i+k+l-1$ distinct vertices and given $\ga$,
if $j\ge1$, then the path $\gb$ is determined 
by choosing $m+j-2$
vertices; 
if $j=0$ then $b$ lies on $\ga$, so $\ga$ is determined by choosing
$i+k+l-2$ vertices, and then $\gb$ is determined by choosing $m-1$ further
vertices. 
Reasoning as in the proof of \refL{L:Type1} we have 
\begin{equation*}
R_{i,j,k,l,m}=
n^{i+j+k+l+m-3}\Bigpar{1+O\Bigparfrac{(i+j+k+l+m)^2}n}. 
\end{equation*}
Due to our cut-off, we have to have $i+k+l\le L$ and $j+k+m\le L$, but we
may for simplicity here allow also paths $\ga,\gb$ with lengths larger than
$L$; the contribution below from pairs with such $\ga$ or $\gb$ is
$O(c^L)=O(n\qwq)$. 
Summing over all possible configurations gives
\begin{align*}
\sum_{i,j\ge0,\,k,l,m\ge1}&R_{i,j,k,l,m}\Big(\frac cn\Big)^{i+j+k+l+m}
\cdot\left(1-\Big(\frac cn\Big)^{k}\right)\\
&=\frac1{n^3}\cdot\sum_{i,j\ge0,\,k,l,m\ge1}c^{i+j+k+l+m}
 \cdot\left(1-\Big(\frac cn\Big)^{k}\right)+O\Big(\frac1{n^4}\Big)\\
&=\frac1{n^3}\cdot\frac{c^3}{(1-c)^5}+O\Big(\frac1{n^4}\Big).
\qedhere
\end{align*}
\end{proof}

\begin{lemma}\label{L:3}
The sum $\sum|\C(\ia,\ib)|$ over all pairs \gab{} with
$\ga\in\ggl{as}$,
$\gb\in\ggl{sb}$
and \gab{} not of Type 1 or 2 is $\onn4$.
\end{lemma}
\begin{proof}
Consider pairs $\gab$ with some given $\la,\lba,\uv$. 
The path $\ga$, which has $\la-1$ interior vertices, may
be chosen in $\le n^{\la-1}$ ways. The $2\uv$ endvertices of the $\mu$ 
subpaths of 
$\gb\setminus\ga$ 
are either $b$ or lie on $\ga$, and given $\ga$, these may be chosen (in order) in $\le (\la+2)^{2\uv}$ ways. 
The $\lba-\uv$ internal vertices in the subpaths can be chosen in $\le
n^{\lba-\uv}$ ways. They can be  
distributed in $\binom{\lba-1}{\uv-1}$ 
(interpreted as 1 if $\uv=\lba=0$) 
ways over the subpaths. 
The path $\gb$ is determined by these endvertices, 
the sequence of
$\lba-\uv$ interior vertices in the subpaths between these endvertices and which vertices belong to which subpath; hence the total number of choices of $\gb$ is
$\le\binom{\lba-1}{\uv-1}(\la+2)^{2\uv}n^{\lba-\uv}$. 

For each such pair $\gab$, we have by \eqref{cov0}--\eqref{cov-}
$|\C(\ia,\ib)|\le (c/n)^{\la+\lba}$. Consequently,
the total contribution  to $\sum|\C(\ia,\ib)|$ from the paths with given
$\la,\lba,\uv$ is at most
\begin{equation}
  \label{c1}
\binom{\lba-1}{\uv-1}(\la+2)^{2\uv}n^{\la-1+\lba-\uv}\parfrac{c}{n}^{\la+\lba}
=\binom{\lba-1}{\uv-1}(\la+2)^{2\uv}c^{\la+\lba} n^{-\uv-1}.
\end{equation}

We consider several different cases and show that each case yields a
contribution $\onn4$, noting that we may assume that
$\lb>\lba$, since otherwise $\ga$ and $\gb$ are edge-disjoint, and thus 
$\C(\ia,\ib)=0$ by \eqref{cov0}.

\begin{romx}
\item {$\uv\ge4$}
Using that $\binom{\lba-1}{\uv-1}\le\lba^{\uv-1}\le L^{\uv}$, and summing \eqref{c1} over $\lba\ge0$ and $\la\le L$, yields for a fixed $\uv$ a contribution
\begin{equation}\label{c2}
\le
(L+2)^{3\uv}(1-c)^{-2}n^{-\uv-1},
\end{equation}
and the sum of these for $\uv\ge4$ is
\begin{equation}
O\bigpar{L^{12}n^{-5}}	
=
O\bigpar{n^{-5}\log^{24} n}	
=\onn4.
\end{equation}

\item{$\uv=3$}
Using that, with $\uv=3$, $\binom{\lba-1}{\uv-1}=\binom{\lba-1}2\le\lba^2$, 
and summing \eqref{c1} 
over all $\la,\lba\ge0$ yields a contribution of at most
\begin{equation}
\sum_{\la,\lba\ge0}\lba^2(\la+2)^6 c^{\la+\lba} n^{-4}
\le
\sum_{\la\ge0}(\la+2)^6 c^{\la}
\sum_{\lba\ge0}\lba^2 c^{\lba} n^{-4}
=\onn4.
\end{equation}
It remains to consider $\uv\le2$.

\item{$\uv=0$}
In this case, $\gb\subset\ga$, and thus
$\lba=0$ and $\la>\lb$ (because $a\neq b$). 
Given $\la$ and $\lb$, we can choose $\gb$ in $\le n^{\lb-1}$ ways and then
$\ga$ in $\le n^{\la-\lb-1}$ ways; for each choice
\eqref{cov-} applies since the edges in $\gb$ have opposite orientations in
$\ga$, and thus the contribution to $\sum|\C(\ia,\ib)|$ is at most
\begin{equation}\label{kaw}
 n^{\lb-1+\la-\lb-1}\parfrac{c}{n}^{\la+\lb}
=
{c}^{\la+\lb}n^{-\lb-2}.
\end{equation}
If $\lb=1$, then $\gab$ is of Type 1, see \refF{F:Type1} ($j=0$).
Since we have excluded such pairs, we
may thus assume that $\lb\ge2$.
Summing \eqref{kaw} over $\la>\lb\ge2$ yields $\onn4$.

\item {$\uv\in\set{1,2}$ and $\ga$ and $\gb$ have some common edge with opposite orientations} \label{l3d}
In this case, \eqref{cov-} applies, and $\binom{\lba-1}{\uv-1}\le\lba\le\lb$. Thus, if we let $\lab=\lb-\lba\ge1$ be the
number of common edges in $\ga$ and $\gb$, then the total contribution 
to $\sum|\C(\ia,\ib)|$ for
given $\la,\lb,\uv,\lab$ (which determine $\lba=\lb-\lab$)
is at most, 
in analogy with \eqref{c1} but using \eqref{cov-}, 
\begin{equation}
  \label{j1}
\lb(\la+2)^{2\uv}n^{\la-1+\lba-\uv}\parfrac{c}{n}^{\la+\lb}
=
(\la+2)^{2\uv}\lb\,c^{\la+\lb} n^{-1-\lab-\uv}.
\end{equation}
For fixed $\uv$, the sum of \eqref{j1} over 
$\la,\lb\ge1$ and $\lab\ge3-\uv$ is $\onn4$, so we
only have to consider $1\le\lab\le2-\uv$. In this case we must have $\uv=1$
and $\lab=1$ (and $\binom{\lba-1}{\uv-1}=1$); thus $\ga$ and $\gb$ have exactly one common edge, which is
adjacent to one of the endvertices of $\gb$. If the common edge is adjacent to
$s$, we have a pair $\gab$ of Type 1, see \refF{F:Type1};
we may thus assume that the common edge is not adjacent to $s$. 
Then, $\lb\ge2$ and the common edge  is adjacent to $b$, which implies 
$b\in\ga$. 
Given $\la$, 
the number of paths $\ga$ that pass through $b$
is $(\la-1)(n-3)_{\la-2}$, since $b$ may be any of the  $\la-1$ interior
vertices. 
The choice of $\ga$ fixes the last interior vertex of $\gb$ (as the
successor of $b$ in $\ga$), and the remaining $\lb-2$ interior vertices
may be chosen in $\le n^{\lb-2}$ ways. The total contribution from this
case is thus at most
\begin{equation}
 (\la-1)n^{\la-2+\lb-2}\parfrac{c}{n}^{\la+\lb}
=
(\la-1)c^{\la+\lb} n^{-4},
\end{equation}
and summing over $\la$ and $\lb$ we again obtain $\onn4$.

\item{$\uv\in\set{1,2}$ and all common edges in $\ga$ and $\gb$ have the same
  orientation}  
The edge in $\gb$ at $s$ does not belong to $\ga$ (since it would have opposite orientation there), so one of the $\uv$ subpaths of $\gb$ outside $\ga$
begins at $s$. If $\uv=1$, or if $\uv=2$ and $b\notin\ga$,
then $\gab$ is of Type 2, see \refF{F:Type2}
($j=0$ and $j\ge1$, respectively). We may thus assume that $\uv=2$ and $b\in\ga$. 
As in case \ref{l3d}, given $\la$,  
we may choose $\ga$ in
$(\la-1)(n-3)_{\la-2}\le \la n^{\la-2}$ ways. The $\uv=2$ subpaths of $\gb$ 
outside $\ga$ 
have 4 endvertices belonging to $\ga$; one is $s$ and the others may be chosen
in $\le \lax^3$ ways. For any such choice, the remaining $\lba-2$ vertices of
$\gb$ may be chosen in $\le n^{\lba-2}$ ways. The total contribution 
for given $\la$ and $\lba$ is thus, using \eqref{cov+},
at most 
\begin{equation}
\lax^{4}n^{\la-2+\lba-2}\parfrac{c}{n}^{\la+\lba}
=
\lax^{4}c^{\la+\lba} n^{-4},
\end{equation}
and summing over all $\la,\lba$ we obtain $\onn4$.
\end{romx}
\end{proof}

\begin{lemma} \label{L:rest} With notation as before, we have
$\C(\XA,\xbb)=\C(\xaa,\XB)=O(n^{-4})$ and
  $\C(\xaa,\xbb)=O(n^{-4})$.
\end{lemma}



\begin{proof}
We only need to consider paths in $\Gamma^L$, which is assumed throughout the proof.
Define $Y_A:=\binom \XA2$, the number of pairs of (distinct) paths from $a$
to $s$, and similarly $Y_B:=\binom {\xb}2$.
Then $0\le \xaa \le Y_A$ and $0\le\xbb \le Y_B$.
Let $\yaa:=Y_A-\xaa$ and $\ybb:=Y_B-\xbb$. Then $\yaa=0$ unless $X_A\ge3$.

Further, let $Z_A:=\binom \XA3$, the number of triples of (distinct) paths from 
$a$ to $s$. Then $0\le \yaa \le Z_A$.

To show that $\C(\XA,\xbb)=\C(\xaa,\XB)=O(n^{-4})$, we write
$\C(\xaa,\XB)=\C(\YA-\yaa,\XB)=\C(\YA,\XB)-\C(\yaa,\XB)$.
Here, $\C(\yaa,\XB)=\E(\yaa\XB)-\E(\yaa)\cdot\E(\XB)$, where
$\E(\yaa\XB)\le\E(\ZA\XB)$, which we will show is $O(n^{-4})$.
Further we will show that $\E(\XA)=\E(\XB)=O(n^{-1})$ and that
$\E(\yaa)\le\E(\ZA)=O(n^{-3})$, so that $\C(\yaa,\XB)=O(n^{-4})$. 
Finally we will show
that $\C(\YA,\XB)=O(n^{-4})$ finishing the proof of the first part of the
lemma. 

For the second part we write $\C(\xaa,\xbb)=\E(\xaa\xbb)-\E(\xaa)\cdot\E(\xbb)$.
We prove that $\E(\xaa\xbb)\le\E(\YA\YB)=O(n^{-4})$ and that $\E(\xaa)=\E(\xbb)\le\E(\YA)=O(n^{-2})$, which finishes the proof.\\[6pt]
\textbf{(i)} $\E(\XA)=O(n^{-1})$:\\[6pt]  
Let $\alpha$ denote an arbitrary path from $a$ to $s$ (in $\Gamma^L$) with length $l\ge1$. Then,
\[\E(\XA)=\E\Big(\sum_{\alpha}I_{\alpha}\Big)=\sum_{\alpha}\E(I_{\alpha})\le\sum_{l=1}^L n^{l-1}\left(\frac cn\right)^{l}\le\frac c{1-c}\cdot n^{-1}=O(n^{-1}).\]

\medskip\noindent
\textbf{(ii)} $\E(\YA)=O(n^{-2})$:\\[6pt]  
Let $\alpha_1$ and $\alpha_2$, with lengths $l_1$ and $l_2$ be two distinct paths from $a$ to $s$. 
Further, let $\delta=|\alpha_2\setminus\alpha_1|$ be the number of edges in $\alpha_2$ not in $\alpha_1$, which form $\mu>0$ subpaths of $\alpha_2$ with no interior vertices in common with $\alpha_1$. 
The number of choices for $\alpha_2$ is (compare the proof of Lemma \ref{L:3}) at most $n^{\delta-\mu}(l_1+1)^{2\mu}\binom{\delta-1}{\mu-1}$, which gives
\begin{align*}\E(\YA)&=\sum_{\alpha_1\ne\alpha_2}\E(I_{\alpha_1}I_{\alpha_2})\le\sum_{l_1,\delta,\mu}n^{l_1-1}\left(\frac cn\right)^{l_1}n^{\delta-\mu}(l_1+1)^{2\mu}\binom{\delta-1}{\mu-1}\cdot\left(\frac cn\right)^{\delta}\\
&=\sum_{l_1,\delta,\mu}n^{-\mu-1}(l_1+1)^{2\mu}c^{l_1+\delta}\binom{\delta-1}{\mu-1}.\end{align*}
\emph{Case 1:} $\mu\ge2$.\\ 
Here, $(l_1+1)^{2\mu}\le(L+1)^{2\mu}$, $\binom{\delta-1}{\mu-1}\le(\delta-1)^{\mu-1}\le\delta^{\mu}\le L^{\mu}$, so that 
the terms are at most $n^{-\mu-1}c^{l_1+\delta}(L+1)^{3\mu}$.
Summing over $l_1$ and $\delta$ gives at most $\frac{c^2}{(1-c)^2}(L+1)^{3\mu}n^{-\mu-1}$, which summed for $\mu\ge2$ is
$O(L^6n^{-3})=O(n^{-3}\log^{12}n)=O(n^{-2})$.\\[3pt]
\emph{Case 2:} $\mu=1$.\\
 Here, $\binom{\delta-1}{\mu-1}=1$, and 
\[\sum_{l_1,\delta}\E(I_{\alpha_1}I_{\alpha_2})
\le n^{-2}\sum_{l_1\ge1}(l_1+1)^2c^{l_1}\sum_{\delta\ge1}c^{\delta}
= O(n^{-2}).\]

\medskip\noindent
\textbf{(iii)} $\E(\ZA)=O(n^{-3})$:\\[6pt]  
We have
\[\E(\ZA)=\sum_{\alpha_1,\alpha_2,\alpha_3}\E(I_{\alpha_1}I_{\alpha_2}I_{\alpha_3}),\]
where $\alpha_1$, $\alpha_2$ and $\alpha_3$ denote three distinct paths from
$a$ to $s$. 

Let $l_1$ denote the length of $\alpha_1$, let
$\delta_2=|\alpha_2\setminus\alpha_1|$ be the number of edges in $\alpha_2$
not in $\alpha_1$ forming $\mu_2>0$ subpaths of $\alpha_2$ intersecting $\alpha_1$ only at the endvertices, and let
$\delta_3=|\alpha_3\setminus(\alpha_1\cup\alpha_2)|$ be the number of edges
in $\alpha_3$ not in $\alpha_1$ or $\alpha_2$ forming $\mu_3\ge0$
subpaths of $\alpha_3$ whose interior vertices are not in $\alpha_1$ or $\alpha_2$.  
Note that $\mu_3=0$ is possible if $\mu_2\ge2$, as
then $\alpha_3$ can be formed by one part from $\alpha_1$ and one part from
$\alpha_2$; however, if $\mu_2=1$ then $\mu_3\ge1$. 
Hence, $\mu_2+\mu_3\ge2$.

If all common edges of the three paths have the same direction,
$\E(I_{\alpha_1}I_{\alpha_2}I_{\alpha_3})=\left(\frac
cn\right)^{l_1+\delta_2+\delta_3}$, otherwise it is 0, so we need only consider paths with the same direction. 
The number of choices for $\alpha_2$ is, as in (ii), at most
$n^{\delta_2-\mu_2}\cdot(l_1+1)^{2\mu_2}\cdot\binom{\delta_2-1}{\mu_2-1}$
and the number of choices for $\alpha_3$ is at most 
$n^{\delta_3-\mu_3}\cdot(l_1+\delta_2-\mu_2+1)^{2\mu_3}\cdot\binom{\delta_3-1}{\mu_3-1}\cdot2^{\mu_2}$, where the last factor is an upper bound for the possible number of choices between segments of $\alpha_1$ and $\alpha_2$. 
Thus, with summation over
$l_1\ge1,\delta_2\ge\mu_2\ge1,\delta_3\ge\mu_3\ge0$,
with  $\mu_2+\mu_3\ge2$, 
\begin{equation}\label{siii}
  \begin{split}
\sum &\E(I_{\alpha_1}I_{\alpha_2}I_{\alpha_3})\le\sum n^{l_1-1}\cdot n^{\delta_2-\mu_2}\cdot(l_1+1)^{2\mu_2}\cdot\tbinom{\delta_2-1}{\mu_2-1}\cdot\\ 
&\cdot n^{\delta_3-\mu_3}\cdot(l_1+\delta_2-\mu_2+1)^{2\mu_3}\cdot\tbinom{\delta_3-1}{\mu_3-1}\cdot2^{\mu_2}\cdot\left(\frac cn\right)^{l_1+\delta_2+\delta_3}\\
&=\sum n^{-\mu_2-\mu_3-1}\cdot(l_1+1)^{2\mu_2}\cdot\tbinom{\delta_2-1}{\mu_2-1}\cdot(l_1+\delta_2-\mu_2+1)^{2\mu_3}\cdot\tbinom{\delta_3-1}{\mu_3-1}\cdot2^{\mu_2}\cdot c^{l_1+\delta_2+\delta_3}.	
  \end{split}
\end{equation}
\emph{Case 1:} $\mu_2+\mu_3\ge3$.\\ 
Here, $(l_1+1)^{2\mu_2}\le(L+1)^{2\mu_2}$, $\tbinom{\delta_2-1}{\mu_2-1}\le
L^{\mu_2}$, 
$(l_1+\delta_2-\mu_2+1)^{2\mu_3}\le(2L+1)^{2\mu_3}\le (L+1)^{3\mu_3}$ 
(assuming as we may $L\ge4$),
$\tbinom{\delta_3-1}{\mu_3-1}\le L^{\mu_3}$ and $2^{\mu_2}\le L^{\mu_2}$, so
that the sum over $l_1,\delta_2, \delta_3$ is at most 
\begin{equation}
  \label{siiia}
n^{-\mu_2-\mu_3-1}\cdot (L+1)^{4\mu_2+4\mu_3}\cdot\sum c^{l_1+\delta_2+\delta_3}
\le (1-c)^{-3}\cdot n^{-\mu_2-\mu_3-1}\cdot(L+1)^{4(\mu_2+\mu_3)}.
\end{equation}
Summing over $\mu_2$ and $\mu_3$, with $\mu_2+\mu_3\ge3$ gives
\[O(n^{-4}\cdot L^{12})=O(n^{-4}\log^{24}n)=O(n^{-3}).\]
\emph{Case 2:} $\mu_2+\mu_3=2$.\\ 
Here, $(\mu_2,\mu_3)\in\{(2,0),(1,1)\}$, so that 
$(l_1+1)^{2\mu_2}\le(l_1+1)^4$, $\tbinom{\delta_2-1}{\mu_2-1}\le\delta_2$, 
$(l_1+\delta_2-\mu_2+1)^{2\mu_3}\le(l_1+\delta_2)^2$,
$\tbinom{\delta_3-1}{\mu_3-1}=1$ and  $2^{\mu_2}\le4$, so that summing over
$l_1,\delta_2, \delta_3$ and $\mu_2+\mu_3=2$ gives at most 
\[2\cdot4\cdot
n^{-3}\cdot\sum_{l_1,\delta_2,\delta_3}(l_1+1)^4\cdot\delta_2\cdot(l_1+\delta_2)^2\cdot
c^{l_1+\delta_2+\delta_3}
=O(n^{-3}).\]

\medskip\noindent
\textbf{(iv)} $\E(\ZA\cdot\XB)=O(n^{-4})$:\\[6pt] 
$\E(\ZA\cdot\XB)=\sum\E(I_{\alpha_1}I_{\alpha_2}I_{\alpha_3}I_{\beta})$, where $
\alpha_1$, $\alpha_2$ and $\alpha_3$ are three distinct paths from $a$ to $s$ and $\beta$ is a path from $s$ to $b$.
We need only consider paths where all common edges have the same direction, as 
$\E(I_{\alpha_1}I_{\alpha_2}I_{\alpha_3}I_{\beta})=0$ otherwise. 

As in (iii) the three $\alpha$ paths are described by
$l_1,\delta_2,\mu_2,\delta_3,\mu_3$. Let
$\delta_4:=|\beta\setminus(\alpha_1\cup\alpha_2\cup\alpha_3)|$ be the number
of edges in $\beta$, not in any of the $\alpha$ paths, and let these form
$\mu_4$ subpaths of $\beta$ whose endvertices lie on
$\alpha_1,\alpha_2,\alpha_3$ but share no other vertices with those
paths. 
The number of choices for the $\alpha$ paths are the same as in (iii) and
given those, and $\delta_4,\mu_4$, the $\beta$ path  can be chosen in at
most 
$n^{\delta_4-\mu_4}\cdot(l_1+\delta_2-\mu_2+\delta_3-\mu_3+1)^{2\mu_4}\cdot\tbinom{\delta_4-1}{\mu_4-1}\cdot3^{2(\mu_2+\mu_3)}$
ways,
where the last factor is a crude upper bound for the number of ways
$\beta$ can choose different sections from the $\alpha$ paths, 
as there are at most $2(\mu_2+\mu_3)$
vertices where a choice can be made
and there are at most 3 possible choices at each of these. 
Clearly, 
$\E(I_{\alpha_1}I_{\alpha_2}I_{\alpha_3}I_{\beta})=(\frac
cn)^{l_1+\delta_2+\delta_3+\delta_4}$ since all common edges have the same
direction.

Note that $\mu_4\ge1$ for non-zero terms as otherwise the
first edge in $\beta$ from $s$ would be the last edge in one of the $\alpha$
paths, and therefore would have opposite direction. Further, $\mu_2\ge1$,
$\mu_3\ge0$, but $\mu_2+\mu_3\ge2$ as $\mu_2=1,\mu_2=0$ would imply that
$\alpha_3=\alpha_1$ or $\alpha_3=\alpha_2$. 

Summing over $l_1\ge1$, $\mu_2\ge1$, $\delta_2\ge\mu_2$, $\mu_3\ge0$,
$\delta_3\ge\mu_3$, $\mu_4\ge1$ and $\delta_4\ge\mu_4$ 
with $\mu_2+\mu_3\ge2$ 
gives at most
\begin{equation}\label{siv}
  \begin{split}
\sum &n^{l_1-1}\cdot n^{\delta_2-\mu_2}\cdot(l_1+1)^{2\mu_2}\cdot\tbinom{\delta_2-1}{\mu_2-1}\cdot n^{\delta_3-\mu_3}\cdot(l_1+\delta_2-\mu_2+1)^{2\mu_3}\cdot\tbinom{\delta_3-1}{\mu_3-1}\cdot2^{\mu_2}\cdot\\
&\cdot n^{\delta_4-\mu_4}\cdot(l_1+\delta_2-\mu_2+\delta_3-\mu_3+1)^{2\mu_4}\cdot\tbinom{\delta_4-1}{\mu_4-1}\cdot3^{2(\mu_2+\mu_3)}\cdot\left(\frac cn\right)^{l_1+\delta_2+\delta_3+\delta_4}\\
&=\sum n^{-\mu_2-\mu_3-\mu_4-1}\cdot(l_1+1)^{2\mu_2}\cdot\tbinom{\delta_2-1}{\mu_2-1}\cdot(l_1+\delta_2-\mu_2+1)^{2\mu_3}\cdot\tbinom{\delta_3-1}{\mu_3-1}\cdot2^{\mu_2}\cdot\\
&\cdot (l_1+\delta_2-\mu_2+\delta_3-\mu_3+1)^{2\mu_4}\cdot\tbinom{\delta_4-1}{\mu_4-1}\cdot3^{2(\mu_2+\mu_3)}\cdot c^{l_1+\delta_2+\delta_3+\delta_4}.	
  \end{split}
\end{equation} 
\emph{Case 1:} $\mu_2+\mu_3+\mu_4\ge4$.\\ 
Here, using the same type of estimates as in (iii) and summing over
$l_1,\delta_2,\delta_3,\delta_4$ gives at most 
\[n^{-\mu_2-\mu_3-\mu_4-1}\cdot(L+1)^{7\mu_2+7\mu_3+4\mu_4}\sum c^{l_1+\delta_2+\delta_3+\delta_4}\le(1-c)^{-4}n^{-\mu_2-\mu_3-\mu_4-1}\cdot(L+1)^{7(\mu_2+\mu_3+\mu_4)},\] 
which summed over $\mu_2+\mu_3+\mu_4\ge4$ is
\[O(n^{-5}\cdot L^{28})=O(n^{-5}\cdot\log^{56}n)=O(n^{-4}).\]
\emph{Case 2:} $\mu_2+\mu_3+\mu_4=3$.\\ 
Here, $(\mu_2,\mu_3,\mu_4)\in\{(2,0,1),(1,1,1)\}$ so that 
$(l_1+1)^{2\mu_2}\le(l_1+1)^{4}$,
$\tbinom{\delta_2-1}{\mu_2-1}\le\delta_2$,
$(l_1+\delta_2-\mu_2+1)^{2\mu_3}\le(l_1+\delta_2)^{2}$,
$\tbinom{\delta_3-1}{\mu_3-1}=\tbinom{\delta_4-1}{\mu_4-1}=1$,
$2^{\mu_2}\le4$,
$(l_1+\delta_2-\mu_2+\delta_3-\mu_3+1)^{2\mu_4}\le(l_1+\delta_2+\delta_3)^{2}$
and $3^{2(\mu_2+\mu_3)}=3^4=81$, so that the sum over
$l_1,\delta_2,\delta_3,\delta_4$ is finite and the total contribution is
$O(n^{-4})$. 

\medskip\noindent
\textbf{(v)} $\E(\YA\cdot\YB)=O(n^{-4})$:\\[6pt] 
$\E(\YA\cdot\YB)=\sum\E(I_{\alpha_1}I_{\alpha_2}I_{\beta_3}I_{\beta_4})$, where 
$\alpha_1$ and $\alpha_2$ are two distinct paths from $a$ to $s$ and $\beta_3$ and $\beta_4$ are two distinct  paths from $s$ to $b$.
As above, we need only consider paths where all common edges have the same direction.
As before, $\alpha_1$ and $\alpha_2$ are described by 
$l_1=|\alpha_1|\ge1$, $\delta_2=|\alpha_2\setminus\alpha_1|\ge1$, 
the number of edges in $\alpha_2$ not in $\alpha_1$, and $\mu_2\ge1$, the
number of subpaths they form that intersect $\alpha_1$ in (and only in) the
endvertices. 
Then $\beta_3$ is described by $\delta_3=|\beta_3\setminus(\alpha_1\cup\alpha_2)|$, the number of edges in $\beta_3$ not in $
\alpha_1$ or $\alpha_2$, and $\mu_3$, the number of subpaths they form with no interior vertices in common with 
$\alpha_1,\alpha_2$.
Similarly, $\beta_4$ is described by
$\delta_4=|\beta_3\setminus(\alpha_1\cup\alpha_2\cup\beta_3)|\ge0$, the
number of edges in $\beta_4$ not in $\alpha_1$, $\alpha_2$ or $\beta_3$
and $\mu_4\ge0$, the number of subpaths they form which intersect $\alpha_1,\alpha_2,\beta_3$ 
in (and only in) the endvertices. 
Note that $\mu_3\ge1$ for every non-zero term,
as otherwise the first edge in $\beta_3$ from $s$ would be the last edge in
one of the $\alpha$ paths, and therefore would have opposite direction.

The number of choices for the $\alpha$ paths are the same as in (ii) and given those, and $\delta_3,\mu_3,\delta_4,\mu_4$, the $\beta$ paths can be chosen in at most
$n^{\delta_3-\mu_3}\cdot(l_1+\delta_2-\mu_2+1)^{2\mu_3}\cdot\tbinom{\delta_3-1}{\mu_3-1}\cdot2^{\mu_2}\cdot
n^{\delta_4-\mu_4}\cdot(l_1+\delta_2-\mu_2+\delta_3-\mu_3+1)^{2\mu_4}\cdot\tbinom{\delta_4-1}{\mu_4-1}\cdot3^{2(\mu_2+\mu_3})$, 
where the last factor is an upper bound for the number of ways $\beta_4$ can
choose different sections from the $\alpha$ paths and $\beta_3$. 

When all common edges have the same direction,
$\E(I_{\alpha_1}I_{\alpha_2}I_{\beta_3}I_{\beta_4})=(\frac
cn)^{l_1+\delta_2+\delta_3+\delta_4}$.
Summing over $l_1\ge1$, $\mu_2\ge1$, $\delta_2\ge\mu_2$, $\mu_3\ge1$, $\delta_3\ge\mu_3$, $\mu_4\ge0$ and $\delta_4\ge\mu_4$ gives at most
\begin{align*}
\sum &n^{l_1-1}\cdot n^{\delta_2-\mu_2}\cdot(l_1+1)^{2\mu_2}\cdot\tbinom{\delta_2-1}{\mu_2-1}\cdot n^{\delta_3-\mu_3}\cdot(l_1+\delta_2-\mu_2+1)^{2\mu_3}\cdot\tbinom{\delta_3-1}{\mu_3-1}\cdot2^{\mu_2}\cdot\\
&\cdot n^{\delta_4-\mu_4}\cdot(l_1+\delta_2-\mu_2+\delta_3-\mu_3+1)^{2\mu_4}\cdot\tbinom{\delta_4-1}{\mu_4-1}\cdot3^{2(\mu_2+\mu_3)}\cdot\left(\frac cn\right)^{l_1+\delta_2+\delta_3+\delta_4}\\
&=\sum n^{-\mu_2-\mu_3-\mu_4-1}\cdot(l_1+1)^{2\mu_2}\cdot\tbinom{\delta_2-1}{\mu_2-1}\cdot(l_1+\delta_2-\mu_2+1)^{2\mu_3}\cdot\tbinom{\delta_3-1}{\mu_3-1}\cdot2^{\mu_2}\cdot\\
&\cdot (l_1+\delta_2-\mu_2+\delta_3-\mu_3+1)^{2\mu_4}\cdot\tbinom{\delta_4-1}{\mu_4-1}\cdot3^{2(\mu_2+\mu_3)}\cdot c^{l_1+\delta_2+\delta_3+\delta_4}.
\end{align*} 
We sum the same terms as in \eqref{siv}, so the sum over all terms with
$\mu_4\ge1$ is $O(n^{-4})$ by the estimates in part {(iv)}.
Hence it suffices to consider the terms with $\mu_4=0$ and thus $\gd_4=0$.
\\
\emph{Case 1:} $\mu_2+\mu_3\ge4$, $\mu_4=0$.\\
Here, each term is $3^{2(\mu_2+\mu_3)}$ times the corresponding term in
\eqref{siii}.
Hence, the estimates in (iii) show that, cf.\ \eqref{siiia}, 
summing over $l_1,\delta_2,\delta_3$ gives at most
\[
(1-c)^{-3}n^{-\mu_2-\mu_3-1}\cdot(L+1)^{6(\mu_2+\mu_3)},
\]
which summed over $\mu_2+\mu_3\ge4$ is
\[O(n^{-5}\cdot L^{24})=O(n^{-5}\cdot\log^{48}n)=O(n^{-4}).\]
\emph{Case 2:} $\mu_2+\mu_3=3$, $\mu_4=0$.\\ 
Here, $\mu_2,\mu_3\le2$ so that 
$(l_1+1)^{2\mu_2}\le(l_1+1)^{4}$, $\tbinom{\delta_2-1}{\mu_2-1}\le\delta_2$,
$(l_1+\delta_2-\mu_2+1)^{2\mu_3}=(l_1+\delta_2)^{4}$,
$\tbinom{\delta_3-1}{\mu_3-1}\le\gd_3$,
$2^{\mu_2}\le4$,
and $3^{2(\mu_2+\mu_3)}=3^6=729$, so that the sum over
$l_1,\delta_2,\delta_3$ is 
$O(n^{-\mu_2-\mu_3-1})$ and the contribution is
$O(n^{-4})$.\\[6pt] 
\emph{Case 3:} $\mu_2+\mu_3=2$, $\mu_4=0$.\\ 
This can only occur if $\mu_2=\mu_3=1$. 
Thus,  $\beta_3$ starts with an edge not in any of the $\alpha$ paths
and, as this is its only excursion it must end up at one of the $\alpha$
paths and follow it to $b$ (if $\beta_3$ were to go straight to $b$ without
coinciding with any of the $\alpha$ paths then $\beta_4$ would have to do
the same, so that $\beta_3=\beta_4$). $\beta_4$ must start as $\beta_3$
until it encounters an $\alpha$ path and must have the possibility to chose
a different path to $b$ than $\beta_3$ along the $\alpha$ paths. This means
that both $\alpha$ paths must pass through $b$ and that they only differ
somewhere between $a$ and $b$. Thus, see Figure \ref{F:3}, there must be
three vertices $x$ (possibly $x=a$), $y$ (possibly $y=x$) and $z$ (possibly
$z=b$) between $a$ and $b$, so that both $\alpha$ paths pass in order
$a,x,y,z,b,s$, and both $\beta$ paths pass in order $s,x,y,z,b$. Both the
two $\alpha$ paths and the two $\beta$ paths follow different subpaths
between $y$ and $z$. 
Let the number of edges between $a$ and $x$ be $i\ge0$, between $x$ and $y$
be $j\ge0$, between $y$ and $z$ be $k\ge1$ and $l\ge1$ for the two
possibilities (with $k+l\ge3$), between $z$ and $b$ be $m\ge0$, between $s$
and $x$ be $r\ge1$ and between $b$ and $s$ be $t\ge1$.

\begin{figure}[htbp]
\includegraphics[width=7cm]{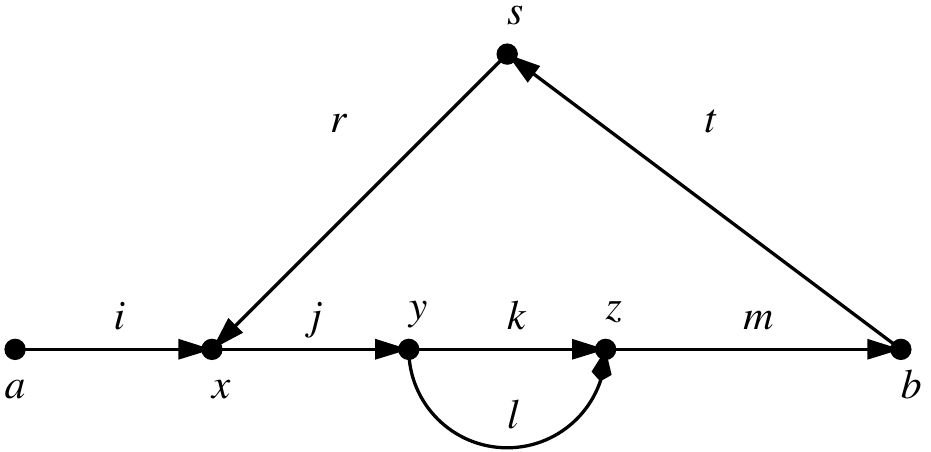}
\caption{Configurations for \emph{Case 3} of  \textbf{(v)}: $\mu_2+\mu_3+\mu_4=2$.} 
\label{F:3}
\end{figure}

Then, $\E(I_{\alpha_1}I_{\alpha_2}I_{\beta_3}I_{\beta_4})=\left(\frac cn\right)^{i+j+k+l+m+r+t}$ and the number of possibilities is at most
$2n^{i+j+k+l+m+r+t-4}$, so that the sum over $i,j,k,l,m,r,t$ is $O(n^{-4})$.

\medskip\noindent
\textbf{(vi)} $\C(\YA,\XB)=O(n^{-4})$:
\[|\C(\YA,\XB)|=|\sum_{\alpha_1\ne\alpha_2}\sum_{\beta}\C(I_{\alpha_1}\cdot I_{\alpha_2},I_{\beta})|\le\sum_{\alpha_1\ne\alpha_2}\sum_{\beta}|\C(I_{\alpha_1}\cdot I_{\alpha_2},I_{\beta})|,\]
where
\[\C(I_{\alpha_1}\cdot I_{\alpha_2},I_{\beta})=\E(I_{\alpha_1}\cdot I_{\alpha_2}\cdot I_{\beta})-\E(I_{\alpha_1}\cdot I_{\alpha_2})\cdot\E(I_{\beta}),\]
which is 0 if $\alpha_1$ and $\alpha_2$ have a common edge with opposite
directions, or if $\beta$ has no edge in common with the $\alpha$ paths.\\ 
Let as above $\alpha_1$ have length $l_1$, $\alpha_2$ have $\delta_2$ edges
not in $\alpha_1$ forming $\mu_2$ subpaths of $\alpha_2$ intersecting
$\alpha_1$ in (and only in) the endvertices.  Let also $\beta$ have
length $l_{\beta}$ with $\delta_3$ edges not in $\alpha_1$ or $\alpha_2$
forming $\mu_3$ subpaths of $\beta$ intersecting $\alpha_1, \alpha_2$ in
(and only in) the endvertices.  Then, if all common edges of $\beta$ and
$\alpha_1\cup\alpha_2$ have the same direction, 
\[|\C(I_{\alpha_1}\cdot I_{\alpha_2},I_{\beta})|=
\left|\left(\frac cn\right)^{l_1+\delta_2+\delta_3}-\left(\frac cn\right)^{l_1+\delta_2+l_\beta}\right|\le\left(\frac cn\right)^{l_1+\delta_2+\delta_3},\]
and if $\beta$ has at least one common edge in opposite direction,
\[|\C(I_{\alpha_1}\cdot I_{\alpha_2},I_{\beta})|=
\left(\frac cn\right)^{l_1+\delta_2+l_\beta}\le\left(\frac cn\right)^{l_1+\delta_2+\delta_3.}
\]
The number of ways of choosing $\alpha_1$, $\alpha_2$ and $\beta$ is at
most, as in (iii) above, 
\[n^{l_1-1}\cdot n^{\delta_2-\mu_2}\cdot(l_1+1)^{2\mu_2}\cdot\tbinom{\delta_2-1}{\mu_2-1}\cdot n^{\delta_3-\mu_3}\cdot(l_1+\delta_2-\mu_2+1)^{2\mu_3}\cdot\tbinom{\delta_3-1}{\mu_3-1}\cdot4^{2\mu_2}.\]
The last factor is $4^{2\mu_2}$  in this case as $\beta$ can have opposite direction in the common subpaths. If there is a crossing between $\alpha_1$ and $\alpha_2$ there may be 4 choices for $\beta$ and there are at most $2\mu_2$ such vertices.
Thus, 
\begin{align*}
\hskip2em&\hskip-2em\sum_{\alpha_1\ne\alpha_2}\sum_{\beta}|\C(I_{\alpha_1}\cdot I_{\alpha_2},I_{\beta})|\le\sum_{l_1,\mu_2,\delta_2,\mu_2,\delta_3}n^{l_1-1}\cdot n^{\delta_2-\mu_2}\cdot(l_1+1)^{2\mu_2}\cdot\tbinom{\delta_2-1}{\mu_2-1}\cdot\\
&\cdot n^{\delta_3-\mu_3}\cdot(l_1+\delta_2-\mu_2+1)^{2\mu_3}\cdot\tbinom{\delta_3-1}{\mu_3-1}\cdot4^{2\mu_2}\cdot\left(\frac cn\right)^{l_1+\delta_2+\delta_3}\\
&\le\sum
n^{-\mu_2-\mu_3-1}\cdot(l_1+1)^{2\mu_2}\cdot\tbinom{\delta_2-1}{\mu_2-1}\cdot(l_1+\delta_2-\mu_2+1)^{2\mu_3}\cdot\tbinom{\delta_3-1}{\mu_3-1}\cdot4^{2\mu_2}\cdot
c^{l_1+\gd_2+\gd_3}.
\end{align*}
Here, $l_1\ge1$, $\mu_2\ge1$, $\delta_2\ge\mu_2$, $\mu_3\ge0$ and
$\delta_3\ge\mu_3$. Note that the terms in the final sum are the same as in
\eqref{siii}, except that $2^{\mu_2}$ is replaced by $4^{2\mu_2}$.
\\[6pt]
\emph{Case 1:} $\mu_2+\mu_3\ge4$.\\ 
Here, using the same estimates as in (iii), 
see \eqref{siiia}, 
the sum over $l_1, \gd_2,\gd_3$ is, for $L\ge16$, at most
\begin{align*}
(1-c)^{-3}\cdot n^{-\mu_2-\mu_3-1}\cdot(L+1)^{4(\mu_2+\mu_3)}.\end{align*}
Summing over $\mu_2+\mu_3\ge4$ gives
$O(n^{-5}\cdot L^{16})=O(n^{-5}\log^{32}n)=O(n^{-4})$.\\[6pt]
\emph{Case 2:} $\mu_2+\mu_3=3$.\\ 
Here, $(\mu_2,\mu_3)\in\{(3,0),(2,1),(1,2)\}$ and
$(l_1+1)^{2\mu_2}\le(l_1+1)^6$, 
$\tbinom{\delta_2-1}{\mu_2-1}\le\delta_2^2$, 
$(l_1+\delta_2-\mu_2+1)^{2\mu_3}\le(l_1+\delta_2)^4$,
$\tbinom{\delta_3-1}{\mu_3-1}\le\delta_3$ and $4^{2\mu_2}\le4^6=4096$. 
Summing over $l_1,\delta_2,\mu_2,\delta_3,\mu_3$ gives at most
\[3n^{-4}\sum_{l_1,\delta_2,\delta_3}4096\cdot(l_1+1)^6\cdot\delta_2^2\cdot
(l_1+\delta_2)^4\cdot \delta_3\cdot c^{l_1+\delta_2+\delta_3}=O(n^{-4}).\] 
\emph{Case 3:} $\mu_2=\mu_3=1$.\\
We need only consider the situation when $\beta$ has at least one edge in common with $\alpha_1\cup\alpha_2$, as otherwise the covariance is 0.\\
\emph{Subcase 3.1: At least one common edge has opposite direction.}\\
$|\C(I_{\alpha_1}\cdot I_{\alpha_2},I_{\beta})|=c^{l_1+\delta_2+l_{\beta}}\cdot n^{-l_1-\delta_2-l_{\beta}}$.
Here, $l_{\beta}\ge2$, as $l_{\beta}=1$ would imply that $\mu_3=0$. Further,
$l_1+\delta_2\ge3$, as otherwise $\alpha_1=\alpha_2$. Let
$l_{\alpha\beta}=|\gb\cap(\ga_1\cup\ga_2)|=l_{\beta}-\delta_3\ge1$.
Then, estimating the number of possible choices of the paths as above,
\begin{align*}
\sum_{l_1,\delta_2,\delta_3,l_{\beta}}&|\C(I_{\alpha_1}\cdot I_{\alpha_2},I_{\beta})|\\&\le\sum n^{l_1-1}\cdot n^{\delta_2-1}\cdot(l_1+1)^2\cdot n^{\delta_3-1}\cdot(l_1+\delta_2)^2\cdot2\cdot c^{l_1+\delta_2+l_{\beta}}\cdot n^{-l_1-\delta_2-l_{\beta}}\\
&=2\cdot\sum_{l_1,\delta_2,\delta_3,l_{\alpha\beta}}(l_1+1)^2\cdot(l_1+\delta_2)^2\cdot c^{l_1+\delta_2+\delta_3+l_{\alpha\beta}}\cdot n^{-3-l_{\alpha\beta}}=O(n^{-4}).
\end{align*}
\emph{Subcase 3.2: All common edges have the same direction.}\\
The first edge of $\beta$, from $s$, must be disjoint with
$\alpha_1\cup\alpha_2$. Let $\beta$ start with $i\ge1$ disjoint steps and
then join one of the $\alpha$ paths, $\alpha_1$ say, for a further $j\ge1$
steps to $b$. Further, let $\alpha_1$ have $k\ge0$ steps before joining
$\beta$ and ending with $l$ steps from $b$ to $s$. As before, $\alpha_2$ is 
determined by two vertices on $\alpha_1$ and $\delta_2-1$ exterior vertices
giving at most $(l_1+1)^2\cdot n^{\delta_2-1}$ possibilities. Further,
$\beta$ can join either of the $\alpha$ paths, and may then do an excursion
along the other path, giving at most 4 possibilities.
Then, as $l_1=k+j+l$, 
\begin{align*}
\sum&|\C(I_{\alpha_1}\cdot I_{\alpha_2},I_{\beta})|\\
&\le4\cdot\sum_{i\ge1}\sum_{k\ge0}\sum_{j\ge1}\sum_{l\ge1}\sum_{\delta_2\ge1}n^{i-1}\cdot
n^{k+j+l-2}\cdot(l_1+1)^2\cdot n^{\delta_2-1}\cdot\left(\frac
cn\right)^{i+k+j+l+\delta_2}\\ 
&=4n^{-4}\cdot\sum_{i,k,j,l,\gd_2}(k+j+l+1)^2\cdot c^{i+k+j+l+\delta_2}=O(n^{-4}).
\end{align*}
\emph{Case 4:} $\mu_3=0$, $\mu_2\in\{1,2\}$.\\ 
$\mu_3=0$ implies that $\beta\subset(\alpha_1\cup\alpha_2)$, so that the
first edge in $\beta$ has opposite direction in
$\alpha_1\cup\alpha_2$. 
Furthermore, at least one of the $\alpha$ paths, $\alpha_1$
say, must pass through $b$, so that $l_1\ge2$. $\alpha_2$ can be chosen in
at most $(l_1+1)^{2\mu_2}\cdot n^{\delta_2-\mu_2}$ ways and there are at
most $2^{\mu_2}$ 
ways for $\beta$ to choose between the $\alpha$ paths,
giving at most 
$n^{l_1-2}\cdot(l_1+1)^{2\mu_2}\cdot n^{\delta_2-\mu_2}\cdot2^{\mu_2}\le4\cdot(l_1+1)^4\cdot n^{l_1+\delta_2-\mu_2-2}$ ways of choosing $\alpha_1$, $\alpha_2$ and $\beta$. The covariance is $-\left(\frac cn\right)^{l_1+\delta_2+l_{\beta}}$. 
Summing over $l_1\ge2$, $\mu_2=1,2$, $\delta_2\ge\mu_2$ and $l_{\beta}\ge1$ gives
\begin{align*}
\sum|\C(I_{\alpha_1}\cdot I_{\alpha_2},I_{\beta})|&\le4\sum(l_1+1)^4\cdot n^{l_1+\delta_2-\mu_2-2}\cdot\left(\frac cn\right)^{l_1+\delta_2+l_{\beta}}\\
&=4\sum(l_1+1)^4\cdot c^{l_1+\delta_2+l_{\beta}}\cdot n^{-\mu_2-l_{\beta}-2}=O(n^{-4}),
\end{align*}
which finishes the proof.
\end{proof}

\end{document}